\documentclass[a4paper]{amsart}
\usepackage{common}

\title[Forward-in-time global exponential integrability]{Domains with globally exponentially integrable parabolic forward-in-time BMO}
 
\author{Kim Myyryl\"ainen}
\address{Department of Mathematics and Statistics, University of Jyv\"askyl\"a, PO Box 35, FI-40014 Jyv\"askyl\"a, Finland}
\email{kim.k.myyrylainen@jyu.fi}

\author{Tuomas Oikari}
\email{tuomas.oikari@gmail.com}

\author{Olli Saari}
\address{Olli Saari, Departament de Matem\`atiques,
	Universitat Polit\`ecnica de Catalunya,
	Avinguda Diagonal 647, 08028 Barcelona,
	Catalunya, Spain and Institute of Mathematics, Universitat Polit\`ecnica de Catalunya, Pau Gargallo 14, 08028 Barcelona, Catalunya, Spain}

\address{Centre de Recerca Matem\`atica, Edifici C, Campus Bellaterra, 08193 Bellaterra, Catalunya, Spain}

\date{\today}

\begin{document}

\begin{abstract}
We characterize those open sets of the space time in which parabolic forward-in-time BMO functions are in a certain forward-in-time exponential integrability class.
The characterization holds under qualitative connectivity assumptions on the domain and is formulated in terms of a quantitative growth bound on a 
forward-in-time version of the classical quasihyperbolic distance.
\end{abstract}

\maketitle

\section{Introduction} 

 \subsection*{Motivation and main result}
 The well-known John--Nirenberg inequality \cite{JohnNirenberg61} asserts that functions of bounded mean oscillation (BMO) are locally exponentially integrable.
 In an open and connected set $O \subset \R^{n}$,
 even more can be said.
 Given the seminorms  
 \[
 \no{f}_{\BMO_{\theta}(O)} := \sup_{\substack{x \in O,r > 0  \\ \dist(x,O^{c}) \ge \theta r }} \ \inf_{c \in \R} \frac{1}{r^{n}} \int_{B(x,r)} |f(y) - c| \, dy ,
 \]  
 a theorem of Reimann and Rychener \cite{MR511997} shows that $\no{f}_{\BMO_{1}(O)} \le C_{n,\theta} \no{f}_{\BMO_{\theta}(O)}$ 
 and it is a result by Smith and Stegenga \cite{SmiSte91b} that finite $\BMO_1(O)$-norm implies global exponential integrability if and only if $O$ satisfies the quasihyperbolic boundary condition.
 The domains satisfying the quasihyperbolic boundary condition are also known as H\"older domains because their planar version can be described in terms of the Riemann mapping having a H\"older continuous extension \cite{BecPom82}, see also \cite{MR802481} for a higher dimensional version.
 We refer to \cites{Buck99, Sta06} for analogous and related results in metric spaces and in spaces of homogeneous type.
 
 We study functions in parabolic (forward-in-time) BMO classes.
 These classes
 are strictly larger than the corresponding parabolic time-symmetric BMO spaces 
 defined using the structure of a space of homogeneous type coming from the standard Lebesgue measure of the space time and the parabolic metric 
 defined below.
For many purposes, especially
in the study of linear uniformly parabolic or doubly nonlinear $p$-parabolic partial differential equations,
 the forward-in-time setup is more convenient than the time-symmetric one.
For example, we mention that the parabolic forward-in-time BMO condition is satisfied by supersolutions to the heat equation,
 something that is not true with the time-symmetric parabolic BMO.
 This makes the forward-in-time class the right tool in the parabolic Moser iteration (see \cites{Moser64,Moser67} and \cite{Tru67}) 
 and in the study of global integrability of supercaloric functions (see \cite{Saa16} and the related work \cite{MR4159822}).

In the present paper,
we focus on the following question:
 \emph{What are the rough domains where forward-in-time parabolic BMO functions are forward-in-time globally exponentially integrable?} 
  The state-of-the-art is the cylindrical result from \cite{Saa16} (sufficiency) and \cite{MR4657417} (necessity), where this question is resolved on domains of product form $I\times\Omega$, where $I\subset\mathbb{R}$ is an interval and $\Omega\subset\mathbb{R}^n$ is a domain.  
  In the present article, we make the leap from cylindrical domains to general rough domains in the space time. We mention \cite{MR2053754} and \cite{Eng17}, and the references therein,
 for more on rough domains relevant for parabolic partial differential equations.

  By using time-directed Harnack chains, we provide the correct definition of H\"older domains in the parabolic time-directed context. We show that these 
  temporal parabolic H\"older 
  domains fully characterize those open sets of the space time where the parabolic forward-in-time BMO condition implies the global forward-in-time exponential integrability of the positive part of the function -- thus providing a complete answer to the question posed above.
 The precise version of our main result around a fixed time slice $t= t_0$ is as follows. 
 We refer to Sections \ref{sec:definitions-and-preliminaries} and \ref{sec:outline} for the precise definitions.

 \begin{theorem}
	\label{theorem:intro1}
	Let $n \ge 1$ and 
    $\Omega \subset \R^{n+1}$ be an open set. Let $t_0\in\mathbb{R}$.
    \begin{itemize}
        \item If the open set $\Omega$ is $t_0$-FIT H\"older domain, then there exist $C, \eta > 0$ and $s_0<t_0$ such that for all $f \in \pbmo^{+}(\Omega_{s_0}^{+})$ we have
	\begin{equation}
    \label{eq:theorem:intro1}
	    \inf_{c \in \R} \int_{\Omega_{t_0}^{+}} \exp \Bigl( \eta (f(z) - c)_{+} / \no{f}_{\pbmo^{+}(\Omega_{s_0}^{+}) } \Bigr) \, dz \le C.
	\end{equation}
    \item Conversely, if there exists some $s_0<t_0$ so that $\Omega$ is $s_0$-FIT connected and there exist $C, \eta > 0$ so that \eqref{eq:theorem:intro1} holds for all $f \in \pbmo^{+}(\Omega_{s_0}^{+})$, then $\Omega$ is a $t_0$-FIT H\"older domain. 
    \end{itemize}
\end{theorem}

    Another way of formulating our main result is through the union of exponential integrability classes. In Theorem \ref{theorem:intro2} below, we use the qualitative $t_0$-FIT connectivity assumptions for all times $t\in (t_{\Omega}^-,t_{\Omega}^+)$, where $t_{\Omega}^- := \inf \{t:(x,t) \in \Omega \}$ and $t_{\Omega}^+ := \sup \{t:(x,t) \in \Omega \}$,
        to stitch together the quantitative bounds (provided by Theorem \ref{theorem:intro1}) on each fixed time slice $t\in (t_{\Omega}^-,t_{\Omega}^+)$.
\begin{theorem}\label{theorem:intro2} 
Let
$n \ge 1$ and 
$\Omega \subset \R^{n+1}$ be open.
Suppose that for all $t_0 \in (t_{\Omega}^-,t_{\Omega}^+)$
the set $\Omega$ is $t_0$-FIT connected.
Then, the following are equivalent.
\begin{itemize}
    \item For every $t_0 \in (t_{\Omega}^-,t_{\Omega}^+)$ there exists some $s_{t_0} \in (t_{\Omega}^-, t_0)$ and $C_{t_0}, \eta_{t_0} > 0$ so that
for all $f \in \pbmo^{+}(\Omega_{s_{t_0}}^+)$ there holds that
\begin{align}\label{eq:theorem:intro2}
    \inf_{c \in \R} \int_{\Omega_{t_0}^{+}} \exp \Big( \eta_{t_0} (f(z) - c)_{+} / \no{f}_{\pbmo^{+}(\Omega_{s_{t_0}}^+)} \Big) \, dz \le C_{t_0}.
\end{align}
\item For all $t_0 \in (t_{\Omega}^-,t_{\Omega}^+)$ the set $\Omega$ is a $t_0$-FIT H\"older domain.
\end{itemize}
\end{theorem}

 The main contribution of these results is that they confirm our definition of $t_0$-FIT H\"older domain
 to be a correct parabolic analogue of the Euclidean notion of a domain satisfying the quasihyperbolic boundary condition. 
 Unlike in the stationary Euclidean case,
 we do not have (as of now) a canonical continuous object such as the quasihyperbolic metric to work with,
 but have decided to formulate the boundary condition in terms of parabolic Harnack chains ($t_0$-FIT chains),
 similarly to the setting of metric spaces \cite{Buck99}, where the existence of rectifiable curves is not guaranteed. 
 As the time lag inherent to parabolic Harnack chains makes these objects more rigid and harder to modify compared to Harnack chains of the stationary setting,
 the definition is not easy to perturb without falling into a different class of domains.
 Moreover, constructing competitors in order to prove lower bounds is less straightforward than with the Harnack chains for harmonic functions.
 This being said,
 it is an interesting question if there is a geometric, PDE-free and curve based, analogue of the quasihyperbolic metric that would be comparable to the counting function we use in the definition of the $t_0$-FIT H\"older domains.

The proof of Theorem \ref{theorem:intro1} consists of three parts. In the first part we follow the outline of Buckley's argument from \cite{Buck99} of outer layer decay of H-chain domains. As $t_0$-FIT H\"older does not imply H-chain, since $t_0$-FIT H\"older chains cannot be spliced along the entire length of the chain, in contrast to H-chains, the argument in \cite{Buck99} does not directly apply and needs a modification. The second part of the proof completes the sufficiency part along the lines of \cite{SmiSte91b}, now using our new axiomatic $t_0$-FIT H\"older set-up instead of the elliptic-to-parabolic upgrade from \cite{Saa16}. Finally, the third part of the proof is a construction of a parabolic BMO function based on the counting function as in the definition of $t_0$-FIT H\"older domains. Here the above mentioned rigidity and reluctance to being modified of the parabolic Harnack chains is a major difference to the Euclidean stationary counterpart. We use the connection to estimates for the heat kernel (or more generally Barenblatt solutions) to produce lower bounds that we were not able to achieve by geometric means as in the stationary setting.

 \subsection*{Literature on time-directed BMO and related developments}
 By now, the literature on parabolic forward-in-time BMO includes the classical definition \cite{Moser64}
 and the first proof of the John--Nirenberg estimate \cite{Moser67}.
 Alternative proofs can be found in the geometry relevant for the heat equation in \cite{FaGa85},
 for the doubly nonlinear (Trudinger's) equation \cite{MR4657417}, and in the setting of spaces of homogeneous type \cite{Aim88}. 
 
 Tightly related to parabolic BMO,
 there exists the notion of a forward-in-time maximal function \cite{Saa18} and
 Muckenhoupt weights \cites{KiSa17,KinMy24B,MR4905323}.
 These references contain somewhat complete treatment of characterization of weighted boundedness of the forward-in-time maximal function, 
 factorization of weights and the reverse H\"older classes. 
 See also \cite{Oik22curved, LiMaOi24, KYYZ25, CKYYZ25} for related commutators and fractional integrals where the time-orientation of the operator or the weight class plays a crucial role.

 \subsection*{Plan of the paper}
 The structure of the paper is as follows. In Section \ref{sec:definitions-and-preliminaries},
 we introduce all the relevant definitions. 
 In Section \ref{sec:outline}, we give the outline of the proof of Theorem \ref{theorem:intro1}.
 The following sections are devoted to the proofs of the main lemmas described in Section \ref{sec:outline}
 and they depend on each other only through Section \ref{sec:outline}.
 Finally in Section \ref{sec:examples}, we briefly discuss the difference between metric H-chain domains and the forward-in-time domains studied in this paper.
 
 \bigskip
 
 \noindent \textit{Acknowledgments.} 
 Kim Myyryl\"ainen was supported by Charles University PRIMUS/24/SCI/020 and Research Centre program No. UNCE/24/SCI/005. Tuomas Oikari was supported by the Research Council of Finland through Project 358180.
 Olli Saari was supported by the Spanish State Research Agency MCIN/AEI/10.13039/501100011033, Next Generation EU and by ERDF ``A way of making Europe'' through the grants RYC2021-032950-I, PID2021-123903NB-I00 and the Severo Ochoa and Maria de Maeztu Program for Centers and Units of Excellence in R\&D, grant number CEX2020-001084-M.
 	\section{Geometric set-up}
	\label{sec:definitions-and-preliminaries}
	We work in the space time $\R^{n+1} = \R^{n} \times \R$ with points taking the form $(x,t) = (x_1, \ldots, x_n,t)$,
	where the last coordinate is time and the vector of the $n$ first coordinates is the space coordinate.
	Given $y \in \R^{n}$ and $s \in \R$,
	we define the notation for spatial and temporal translations
	\[
	(x,t) + y := (x+y,t), \qquad (x,t) + s := (x,t+s).
	\]
	For the projections to the space and time coordinates, we use the notation  
	\[
	\pi_{S}(x,t) = x, \quad \pi_{T}(x,t) = t.
	\]
	In what follows, 
	we use the letter $D$, possibly with subscripts,
	to denote generic positive constants always having the dependencies as specified in the statements of the proofs in which they appear. 
	We write $A \lesssim B$, $A \gtrsim B$ and $A \sim B$ to denote inequalities that hold up to such a constant.

	Given $p \in (1,\infty)$,
	we define the parabolic distance between two points $z, w \in \R^{n+1}$ as 
	\[
	d(z,w) := \max (|\pi_{S}(z-w)|, |\pi_{T}(z-w)|^{1/p}).
	\] 
	Throughout the paper all distances are with the parabolic distance, unless otherwise specified. 
	The parameter $p$ is fixed and we do not include it in the notation.
	The parabolic distance is a metric and
	together with the $(n+1)$-dimensional Lebesgue measure,
	it endows the space time with the structure of a metric space with a doubling measure.
	A metric ball with respect to the parabolic distance $d$ is called a parabolic cylinder, cylinder for short.
	For a cylinder $P = B_d(z,\rho)$, its radius is denoted by $r(P) = \rho$ and its center by $c(P) = z$.
	By $\lambda P$ we denote the parabolic cylinder with the centre $c(P)$ and radius $\lambda r(P),$ for $\lambda>0.$
	Further, for $a>0$, we denote
	\begin{equation}
		\label{gdef:lags}
		P^{\pm} = P \pm 3 r(P)^{p}, \quad P^{\pm,a} = P \pm a r(P)^{p}, \quad \widetilde{P}^{a} = \conv (P^{+,a} \cup P^{-,a}).
	\end{equation}
		Notice that the time lag between $P^{-,a}$ and $P^{+,a}$ is $(a-1)r(P)^p,$ which is positive if  $a >1.$ 

	We next introduce $\fit_A$ chains (FIT chains for short), which are time-directed versions of Harnack chains.
	\begin{definition}[$\fit_A$ links and $\fit_A$ chains]\label{defn:FITchain}
		Let $A \in  [2,\infty)^{2}$ be fixed.
		\begin{itemize}
			\item 
			An ordered pair $(P,Q)$ of parabolic cylinders is a $\op{FIT}_{A}$ link (forward-in-time link) provided that
			\begin{align*}
				1/A_1\leq \frac{|Q|}{|P|}\leq A_1,\qquad 
				|P^+ \cap Q^- | \geq  \frac{1}{A_2} |Q\cup P|.
			\end{align*}
			\item An $N$-tuple of parabolic cylinders $(Q_j)_{j=1}^N$ is a $\op{FIT}_{A}$ chain 
			provided that each $(Q_j,Q_{j+1})$ with $1\le j \le N-1$ is a $\op{FIT}_{A}$ link. 
			Such a chain is said to connect $Q_1$ to $Q_N$.
		\end{itemize}
	\end{definition}
	
	\begin{rem} Note that by the condition involving $A_2$ it follows that $|P|\geq |P^+ \cap Q^- | \geq |Q\cup P|/A_2 \geq |Q|/A_2$ (and similarly with $P,Q$ swapped) and hence $1/A_2 \leq |P|/|Q| \leq A_2$ so that the condition involving $A_1$ follows with $A_1 = A_2$.  However, for a possible future reference, we have decided to track volume ratio and intersection separately.
	\end{rem}
	As it were,
	a $\fit_A$ chain is a time-directed version of a Harnack chain.
	Harnack chains are used to describe quantitative connectivity properties of open sets in the Euclidean space.
	Similarly, we will use FIT chains to discuss quantitative and time-directed connectivity in the space time. 
	To do so,
	we will next, in the three definitions to follow, relate FIT chains to a reference domain $\Omega.$
	
	\begin{definition}[$\fit_A$-distance between cylinders]
		Let $\Omega \subset \R^{n+1}$ be open, 
		let $A = (A_1,A_2,A_3) \in [2,\infty)^{3}$ and 
		let $(P,Q)$ be an ordered pair of parabolic cylinders contained in $\Omega$.
		\begin{itemize}
			\item 		We define $\mathcal{C}_{\Omega,A}(P,Q)$ as the family of all $\op{FIT}_{(A_1,A_2)}$ chains $C$ connecting $P$ to $Q$ in the sense of Definition \ref{defn:FITchain}
			so that $A_3 \widetilde{R}^{3} \subset \Omega$ for all $R \in C$.
			\item 		We define the $\fit_A$ distance from $P$ to $Q$ in $\Omega$ as
			\begin{align}\label{eq:QHdistPrel}
				k_{\Omega,A}(P,Q) :=  \inf \{ \# C : C \in   \mathcal{C}_{\Omega,A}(P,Q )\},
			\end{align}
			whenever $ \mathcal{C}_{\Omega,A}(P,Q) \ne \varnothing$, and as $k_{\Omega,A}(P,Q) = \infty$, when $  \mathcal{C}_{\Omega,A}(P,Q) = \varnothing$.
		\end{itemize}
	\end{definition}
	
	FIT connectivity with respect to time will be defined as a qualitative property.
	We ask that every point after $t_0$ can be connected to a reference parabolic cylinder before $t_0$ by a $(t_0,A)$-FIT chain.
	Thus, relative to the open set $\Omega$,
	we define the future and past parts
	\begin{align*}
		\Omega_{t_0}^+ = \{ (x,t)\in\Omega : t > t_0 \},\qquad 
		\Omega_{t_0}^- = \{ (x,t)\in\Omega : t < t_0 \}.
	\end{align*}    

	\begin{definition}[$t_0$-FIT connectivity]\label{defn:FITConnected}
		Let $t_0 \in \R$ .
		An open set $\Omega \subset \mathbb{R}^{n+1}$ is said to be $t_0$-FIT connected if 
		there exists a parabolic cylinder $R_{*} = R_{*}(t_0,\Omega)\subset \Omega_{t_0}^-$ such that for all $z\in \Omega_{t_0}^+$ there exists a parabolic cylinder $Q\subset \Omega$ and $A \in [2,\infty)^{3}$ such that $z\in Q^+ \subset \Omega_{t_0}^+$ and $k_{\Omega,A}(R_{*},Q)< \infty.$ 
		The parabolic cylinder $R_{*}$ is said to be the  central cylinder.
	\end{definition}

The parameter $A$ in the definition of $t_0$-FIT connectivity carries no relevant information.
In fact, we may fix $A$ to be $(2,2,2)$ for all points without loss of generality due to the following observation.

\begin{proposition}
\label{prop:varying-chains}
Let $t_0 \in \R$ and $A,A' \in [2,\infty)^{3}$.
Let $\Omega \subset \mathbb{R}^{n+1}$ be open.
Then there exists $\varepsilon = \varepsilon(A,A',n,p) \in (0,1)$ such that the following holds. 

Let $z \in \Omega_{t_0}^{+}$.
If $Q \subset \Omega_{t_0}^{+}$ and $R_* \subset \Omega_{t_0}^{-}$
satisfy $z \in Q^{+}$ and $k_{\Omega,A}(R_{*},Q)< \infty$,
then 
for
$Q' = [B_d(z,\varepsilon r(Q))]^-$
and $R_{*}' = \varepsilon R_{*}$
we have
\[
  k_{\Omega,A'}(R_{*}', Q')  \le c(A,A',n,p)k_{\Omega,A}(R_{*},Q).
\]
\end{proposition}
\begin{proof} 
Fix $\{P_i\}_{i=1}^{N} \in \mathcal{C}_{\Omega,A}(R_{*},Q)$.
For $1 \leq i < N$,
let $z_i$ be the center of $P_{i}$,
and let $z_{N} := z \in P_{N}^{+}$.
Obviously there exist $\varepsilon \in (0,1)$ and $c$, only depending on $n$, $p$, $A$ and $A'$;
so that for all $1 \le i < N$
\[
k_{A_{3}\widetilde{P}_i^{3} \cup A_{3}\widetilde{P}_{i+1}^{3},A'}(B_d( z_i,\varepsilon r(P_i)),B_d( z_{i+1},\varepsilon r(P_{i+1})) ) \le c.
\]
Let 
\[
C_i \in  \mathcal{C}_{{A_{3}\widetilde{P}_i^{3} \cup A_{3}\widetilde{P}_{i+1}^{3},A'}}(B_d( z_i,\varepsilon r(P_i)),B_d( z_{i+1},\varepsilon r(P_{i+1})) )
\]
have length at most $c+1$. 
Since $A_{3}\widetilde{P}_i^{3} \subset \Omega$ for all $i$, 
by the choice of the chain $\{P_i\}_{i=1}^{N}$,
we see that 
\[
C' := C_1 \cup \cdots \cup C_{N-1} \in \mathcal{C}_{\Omega,A'}( B_d( z_1,\varepsilon r(P_1)), B_d( z_N,\varepsilon r(P_N))).
\]
As 
\[
\# C' \le \sum_{i=1}^{N-1} \# C_i \le (N-1)c, 
\]
the claim follows.
\end{proof}

	To make FIT connectivity quantitative, 
	we impose a growth bound on the FIT distance to the central cylinder.
	This is the content of the following definition. 
	
	\begin{definition}[$(t_0,A)$-FIT H\"older domain]\label{defn:FITHolder}
		Let $t_0 \in \R$ and $A \in [2,\infty)^{3}$.
		We say that an open set $\Omega \subset\mathbb{R}^{n+1}$ is a $(t_0,A)$-FIT H\"older domain if the following points hold.
		\begin{itemize}
			\item The set $\Omega$ is $t_0$-FIT connected with a central cylinder $R_{*} \subset \Omega_{t_0}^{-}.$
                \item There exists $K > 0$ 
			such that for all $z \in \Omega_{t_0}^{+}$ there holds that 
				\begin{align}\label{eq:defn:FITHolder}
                \inf_{ z\in Q^+, Q\subset  \Omega^+_{t_0}}
				k_{\Omega,A}(R_{*},Q) \le K \log_2(1+1/\dist_d(z,(\Omega^+_{t_0})^{c})),
			\end{align}
		where $k_{\Omega,A}(R_{*},Q)$ is as defined on the line \eqref{eq:QHdistPrel}.
		\end{itemize}
	\end{definition}
	
\begin{rem}
\label{rmk:chain-varying-hoelder}
By Proposition \ref{prop:varying-chains},
the class of $(t_0,A)$-FIT H\"older domains is independent of the parameter $A \in [2,\infty)^{3}$.
The value of the parameter $K$ in the definition, however, depends on $A$.
For transparency,
we keep track when $A$ has to be changed and we use the notation $(t_0,A)$-FIT H\"older during most of the paper.
We reserve the shorter expression $t_0$-FIT H\"older domain for $(t_0,(2,2,2))$-H\"older domains.
\end{rem}

	The following estimate on optimal lengths of parabolic Harnack chains 
	connecting two parabolic cylinders in free space time will be crucial in what follows.
	The upper bound for the length of the chain is straightforward.
	One computes the length of a FIT chain that adjusts separately the space coordinate, the time coordinate and the scale parameter.
	The argument for the lower bound goes through the parabolic Harnack inequality. 

    	\begin{lemma}
		\label{lemma:Barenblatt}
		Let $A \in [2,\infty)^{2}$, $(x_0,t_0),(y_0,s_0) \in \R^{n+1}$ and $R_1 = B_d((x_0,t_0);r_0)$.
		Let $\mathcal{R}$ be the family of all parabolic cylinders containing $(y_0,s_0)$.
		Then
			\begin{equation*}
				\inf_{R \in \mathcal{R}} \quad \inf_{C \in \mathcal{C}_{A,\R^{n+1}}(R_1,R)} \# C 
				\gtrsim 
				\log \frac{s_0-t_0+2r_0^{p}}{r_0^{p}} + \biggl(\frac{|y_0-x_0|^{p}}{s_0-t_0 + 2r_0^{p}} \biggr)^{\frac{1}{p-1}},
		\end{equation*}
		where the implicit constants depend on $n$, $p$ and $A$.
        
        Furthermore, if $s_0 - t_0 \ge 10 r_0^{p}$, then
        \begin{equation*}
				\inf_{R \in \mathcal{R}} \quad \inf_{C \in \mathcal{C}_{A,\R^{n+1}}(R_1,R)} \# C 
				\lesssim 
				\log \frac{s_0-t_0+2r_0^{p}}{r_0^{p}} + \biggl(\frac{|y_0-x_0|^{p}}{s_0-t_0 + 2r_0^{p}} \biggr)^{\frac{1}{p-1}}.
		\end{equation*}
	\end{lemma}
	
	\begin{proof}
		The claim is invariant under translations of the space time so we may assume $t_0 > 2r_0^{p}$ so that $R_0,R \subset \R^{n} \times (0,\infty)$. 
		To prove the claimed lower bound ($\gtrsim$),
		we first recall that 
		\[
		u(x,t)=t^{\frac{-n}{p(p-1)}}
		e^{-\frac{p-1}p\bigl(\frac{|x|^p}{pt}\bigr)^{\frac1{p-1}}} 
		\]
		is a solution to 
		\begin{equation}
			\label{eq:trudingers-equation}
			\partial_t (u^{p-1}) - c_0 \dive(|\nabla u|^{p-2} \nabla u) = 0
		\end{equation}
		in $\R^{n} \times (0,\infty)$ for some constant $c_0 = c_0(p,n)>0$.
		By positivity of $u$,
		which is obvious from the explicit formula,
		we see that the parabolic Harnack inequality for \eqref{eq:trudingers-equation} (Theorem 2.1 in \cite{KiKu07}) applies.
		Hence for every $R \in \mathcal{R}$ such that there exists $C \in \mathcal{C}_{A,\R^{n+1} }(R_1,R)$,
		we conclude that there exists $c = c(n,p) > 1$ so that 
		\[
		\frac{u(x_0,t_0-r_0^p)}{u(y_0,s_0)} \le c^{\# C}.
		\]
		As both the parabolic Harnack inequality 
		and the number of links in a chain stay invariant under the coordinate change 
		\[
		(x,t) \mapsto \biggl( \frac{x-x_0}{r_0},\frac{t-t_0}{r_0^{p}} + 2 \biggr),
		\]
		we conclude 
		\[
		\# C \gtrsim \log \frac{u(0,1)}{u((y_0 - x_0)/r_0,(s_0-t_0 )/r_{0}^{p} + 2 )}
		\sim \log \biggl(2 + \frac{s_0-t_0 }{r_0^p} \biggr) + \biggl(\frac{|y_0-x_0|^{p}}{s_0-t_0 + 2r_0^{p}} \biggr)^{\frac{1}{p-1}},
		\]
		which is the claimed lower bound.
		The upper bound follows by counting the parabolic cylinders of any reasonable FIT chain in $\mathcal{C}_{A, \R^{n+1}}(R_1,R)$.
	\end{proof}

	\section{The main line of the proof}
	\label{sec:outline}
	
	\begin{definition}
		\label{def:PBMO}
		Let $\Omega \subset \R^{n+1}$ be open.
		Let $a > 1$, $q > 0$ and $\theta \ge 1$.
		We denote (using the notation in \eqref{gdef:lags})
		\begin{equation*}
			\| f \|_{{\pbmo}^{+}(a,q,\theta,\Omega)}
			= \sup_{ \theta \widetilde{P}^{a}  \subset \Omega } \ \inf_{c \in \R} \Big( \fint_{P^{+,a}}(f(z)-c)_{+}^q \, dz+ \fint_{P^{-,a}}(f(z)-c)_{-}^q\, dz \Big)^{1/q},
		\end{equation*} 
		where the supremum is over a family of parabolic cylinders. 
		We set 
		\[
		\no{\cdot}_{\pbmo^{+}(\Omega)} := \no{\cdot}_{{\pbmo}^{+}(3,1,1,\Omega)}
		\]
		and define the class 
		$\pbmo^{+}(\Omega)$ as the family of  those locally integrable functions $f$ such that $\no{f}_{\pbmo^{+}(\Omega)}$ is finite.
	\end{definition}

	\begin{proof}[Proof of Theorem \ref{theorem:intro1}]
		We will cover the domain by Whitney type parabolic cylinders,
		which are then connected to the  central cylinder of the domain by a $(t_0,A)$-FIT chain inside $\Omega$.
		The lengths of these chains blow up when approaching the boundary of the domain.
		Hence the first task is to show that the measures of the layers close to the boundary decay sufficiently fast.
		
		\begin{lemma}
			\label{lemma:OLD}
			Let $A \in [2,\infty)^{3}$, $t_0 \in \R$ and $R > 0$.
			Assume that $\Omega$ is a $(t_0,A)$-FIT H\"older domain. 
			Then, there exist constants $D_1,\alpha > 0$ such that the following holds.
			For an integer $k \ge 0$, 
			denote 
			\begin{align*}
					L_k := \{ z \in \Omega_{t_0}^{+} : 2^{-k-1} < \dist_d(z,(\Omega_{t_0}^+)^{c}) \le 2^{-k}\}.
			\end{align*}
			Then, there holds that 
			\begin{align}\label{eq:OLD:1}
				|L_k | \le D_1 2^{-\alpha k}.
			\end{align}
			In addition,
			\begin{align}\label{eq:OLD:2}
				|\Omega_{t_0}^+| < \infty.
			\end{align}
		\end{lemma}
		
		Lemma \ref{lemma:OLD} is proved in Section \ref{sec:old}.
		Once the decay of the outer layers is known,
		we proceed to prove a level set estimate.
		We estimate the parabolic oscillations in Whitney type cylinders by the local parabolic John--Nirenberg lemma.
		The remaining part of the proof bounds the difference of averages in the covering Whitney cylinders and the  central cylinder of the domain along a $(t_0,A)$-FIT chain using the
		$\pbmo^{+}(\Omega)$ condition,  an
		estimate on the length of the chain, 
		and the outer layer decay from Lemma \ref{lemma:OLD}.
		
		\begin{lemma}
			\label{lemma:global-jn}
			Fix $A \in [2,\infty)^{3},$
            $t_0 \in \R$ and let $\Omega$ be a $(t_0,A)$-FIT H\"older domain.  
            Then there exists constants $D = D(n,p,\Omega^+_{t_0})$ and $\alpha = \alpha(n,p,\Omega^+_{t_0})>0$ such that for all $f \in {\op{PBMO}}^{+}(3,1,A_3,\Omega)$
			and all $\lambda > 0$,
			\[
			\inf_{c \in \mathbb{R}} |\{ z \in \Omega^+_{t_0} : (f(z) - c)^{+} > \lambda \}  
			\le DA_3 
            \exp \biggl(  - \frac{\alpha \lambda}{A_2 D K\no{f}_{\op{PBMO}^+(3,1,A_3,\Omega)}}
			\biggr).
			\]
		\end{lemma}
		
		Lemma \ref{lemma:global-jn} is proved in Section \ref{sec:global}.
		We remark that Lemma \ref{lemma:global-jn} has the following corollary,
		which also follows from the John--Nirenberg inequality in \cite{Aim88} (see also \cite{MR4657417} for a simpler proof) and the local-to-global result in \cite{Saa16}.
		However, instead of quoting it, 
		we can consider its local-to-global part as a corollary of our current proof.
		\begin{corollary}
			\label{cor:equivalences}
			Let $\Omega \subset \R^{n+1}$ be open and let $a > 1$, $q > 0$ and $\theta \ge 1$.
			If $f \in L^{\delta}_{loc}(\Omega)$ for some $\delta > 0$,
			then 
			\[
			\| f \|_{{\pbmo}^{+}(a,q,\theta,\Omega)} \sim_{a,q,\theta,n,p} \| f \|_{\pbmo^{+}(\Omega)}.
			\]
		\end{corollary} 

        Once Lemma \ref{lemma:global-jn} and Corollary \ref{cor:equivalences} are known,
		the implication assuming a $t_0$-FIT condition in Theorem \ref{theorem:intro1} follows by a straightforward application of the Cavalieri principle.
        Indeed, we apply Lemma \ref{lemma:global-jn} with $\Omega_{s_{t_0}}^+$ in place of $\Omega$.
		The converse direction is handled in Section \ref{sec:universal},
		where the proof is concluded by Lemma \ref{lem:necessity}.
    \end{proof}

	\section{Proof of Lemma \ref{lemma:OLD}: Outer layer decay}
	\label{sec:old}
	
	The proof of the outer layer decay Lemma \ref{lemma:OLD} is divided into two parts.
	First,
	we show that a $(t_0,A)$-FIT H\"older domain is almost bounded in the sense that the parts far away must necessarily be very close to the boundary. 
	This is different from the behavior of the H-chain sets studied in \cite{Buck99} 
	as H-chain sets are bounded. See Section \ref{sect:examples:ex4} for an example of an unbounded FIT H\"older domain. We also show that parts far away from the boundary have to be close to the central cylinder $R_*.$
	\begin{lemma} 
		\label{lem:almost-bounded} 
		Let $A \in [2,\infty)^{3}$ and $t_0 \in \R$.
		Then, there exist constants $D_1,D_2 > 0$ such that the following holds.
		Assume that $\Omega$ is a $(t_0,A)$-FIT H\"older domain and $k\geq 1$.
		Then, 
		\begin{align}\label{eq:lem:almost-bounded1}
		\bigl\{ z \in \Omega_{t_0}^{+} : 2^{-k} \le  \dist_d(z,(\Omega_{t_0}^+)^c)  \bigr\} \subset B(c_{R_*},D_1 k^{D_2})
		\end{align}
		for the centre point $c_{R_*} \in R_*$. 
		Moreover, there exists an absolute constant $D_3>0$ so that
		\begin{align}\label{eq:lem:almost-bounded2}
			\bigl\{ z \in \Omega_{t_0}^{+} : 1 \leq  \dist_d(z,(\Omega_{t_0}^+)^c)  \bigr\} \subset B(c_{R_*},D_3).
		\end{align}
	\end{lemma}
	\begin{proof}
		Let $R_*$ be the central cylinder and without loss of generality assume that $c_{R_*} = 0.$ 
		If $B_d(0,1)^{c} \cap \Omega=\emptyset$,
		the claims hold trivially.
		Assume $z \in B_d(0,1)^{c} \cap \Omega_{t_0}^+$. 
		Let $R^+ \ni z$ be a parabolic cylinder and         
        $C \in \mathcal{C}_{\Omega,A}(R_{*},R)$ such that 
        \[
        \#C \sim \inf_{ z\in Q^+, Q\subset  \Omega^+_{t_0}}
        k_{\Omega,A}(R_{*},Q).
        \]
		By $|z|\gtrsim 1,$ Lemma \ref{lemma:Barenblatt} and inequality \eqref{eq:defn:FITHolder},
		we have 
        \begin{align}\label{eq:lem:almost-bounded3}
			1\lesssim
            \log \frac{\pi_T (z)+2r_{R_*}^{p}}{r_{R_*}^{p}} + \biggl(\frac{|\pi_S (z)|^{p}}{\pi_T (z) + 2r_{R_*}^{p}} \biggr)^{\frac{1}{p-1}}
            \lesssim \# C \lesssim \log \biggl(1 + \frac{1}{\dist_{d}(z,(\Omega_{t_0}^+)^c)} \biggr),
		\end{align}
		so that   $\dist_{d}(z,(\Omega_{t_0}^+)^c)\lesssim  1.$
		We conclude from \eqref{eq:lem:almost-bounded3} that if $1 \leq  \dist_d (z,(\Omega_{t_0}^+)^c)$, then $d(z,0) \lesssim 1$ and hence we obtain \eqref{eq:lem:almost-bounded2}.
		
		The bound \eqref{eq:lem:almost-bounded3} tells us that every point far from the origin must be close to the boundary. 
		This implies that no parabolic cylinder admissible in a FIT chain far away can be large, and neither can they be large near the origin.
		Quantitatively,
		when $z \in [B(0,2^{j+1}) \setminus B(0,2^{j})] \cap \Omega_{t_0}^{+}$, for $j \ge 1$, 
		we write a crude estimate for number of uniformly sized cylinders needed to connect an annulus to its center
		\[
		\# C \gtrsim  \dist_d([B(0,2^{j+1}) \setminus B(0,2^{j})],0) \sim 2^{j}.
		\]
		Then, write the upper bound from the $(t_0,A)$-FIT H\"older hypothesis for the almost optimal chain 
		\[
		\# C \lesssim  \frac{1}{\dist_{d}(z,(\Omega_{t_0}^+)^c)}
		\]
		so that altogether
		\[ 
		d(z,0) \sim 2^{j} \lesssim \# C \lesssim \log\biggl( \frac{1}{\dist_{d}(z,(\Omega_{t_0}^+)^c)}\biggr).
		\]
		Now if $2^{-k} \le  \dist_d(z,(\Omega_{t_0}^+)^c)$
		we see that 
		\[
		d(z,0) \le k^{D_2},
		\]
		as was claimed. 
	\end{proof}

	Second,
	we show that there is a macroscopic top scale such that below that scale we have enough interior corkscrews on average, provided we are far enough in $\Omega_{t_0}^+$ in the time variable relative to the scale under scrutiny. 
	Informally, this argument is similar to that in \cite{Buck99}, but due to FIT chains exiting $\Omega_{t_0}^+$ an additional argument has to be made to cover the cases where the chain begins from near the set $\{t=t_0\}$. 
	For the following lemma, consider a fixed constant $M>1$ and denote 
	\[
	L_k^+ := \bigl\{ z \in \Omega_{t_0}^{+} :   2^{-(k+1)} < \dist_{d}(z,(\Omega_{t_0}^+)^c) \le 2^{-k}  \bigr\} \cap \bigl\{ (x,t) : t > t_0 + 2^{-pk/M} \bigr\}.
	\]
    Notice that if $z\in L_k^+$, then $\dist_p(z,\{t\leq t_0\}) > 2^{-k/M}.$
	We fix $M=2$ but preserve the variable to make it easier to track in the following proofs.
	\begin{lemma} 
		\label{lem:MeanCorkscrews} 
		Let $A \in [2,\infty)^{3}$, $t_0 \in \R$ and $\Omega$ be a $(t_0,A)$-FIT H\"older domain.  
		Then, there exist constants $D_1,D_2 > 0$ such that: for each integer $k\geq 1$, there exists a family of parabolic cylinders $\mathcal{F}(k)$ such that 
		\[
		\bigcup_{P\in\mathcal{F}(k)}P \subset \Omega_{t_0}^+,\qquad  \sum_{P \in \mathcal{F}(k)} 1_{P} \le 1, \qquad  \sum_{P \in \mathcal{F}(k)}  D_{2} 1_{D_1 P}  \ge k1_{L_k^+}.
		\]
	\end{lemma}
	\begin{proof} 
		By Lemma \ref{lem:almost-bounded} it is enough to prove the claim for $k\geq k_0$, for some absolute $k_0\in \mathbb{Z}$ that is allowed to depend on the constants in the definition of the fixed $t_0$-FIT H\"older domain. Indeed, if $k\leq k_0$, then we take any fixed cylinder $P\subset \Omega_{t_0}^+$ and by Lemma \ref{lem:almost-bounded} there exists some absolute constant $\sigma>0$ so that $L_k \subset \sigma P$ for all $k\geq k_0$. In this case we set $D_1 = \sigma$ and $D_2 = k_0$ and $\mathcal{F}(k)= \{P\}$.
		In what follows we will quantify $k_0$
        providing the desired family $\mathcal{F}(k)$ for each $k \geq k_0$ and with some constants $D_1,D_2$ not depending on $k$. 
		
		Let $R_{*}$ be the  central cylinder provided by the $(t_0,A)$-FIT H\"older hypothesis.
		Pick a point $z \in L_k^+$, 
		and let $P_N$ be a parabolic cylinder with $z \in P_N^{+}$ and such that $C_z := \{P_{i}\}_{i=1}^{N} \in \mathcal{C}_{\Omega,A}(R_{*},P_N)$ is a minimal chain satisfying 
		\begin{equation}
			\label{eq:corkproof1}
			N \le 2 K \log_{2}\biggl(1+ \frac{1}{\dist_d(z,(\Omega^+_{t_0})^{c})} \biggr) \le 4 K k.
		\end{equation}
		This is possible by Definition \ref{defn:FITHolder}. 
		
		\textit{Step 1.} 
		 Denote $c_i := c(P_i)$ and $r_i := r(P_i)$, for $i=1,\dots ,N.$
		We show that annuli centred at $z$ and growing with a sufficiently large absolute geometric factor (depending only on the $t_0$-FIT H\"older assumption) and not containing the central cylinder $R_*$ contain at least a single centre point $c_i$. 
		
		The chain $C_z$ is minimal so $z \notin P_{i}$,
		and as the chain is $\fit_{A_1,A_2}$,
		we have for all $i \in [2,N]$ that
		\begin{align}\label{eq:corks1}
			d(c_i, c_{i-1}) \le 7r_i + 7r_{i-1} \le 7 (1+ A_1) r_i \le 7 (1+ A_1) d(c_i,z)
		\end{align}
		and further that
		\begin{align}\label{eq:corks2}
			d(c_{i-1}, z) \le d(c_{i-1},c_{i}) + d(c_{i},z) \le (8+7A_1) d(c_i,z).
		\end{align}
		Denote $B_j = B_d(z, (8+7A_1)^{j} r_{N})$. We want to find a centre point in the annulus $ B_{j}\setminus B_{j-1}$.
		
		Let $j \ge 2$ be such that $(8+7A_1)^{j} r_N \le d(c_1,z)$ and $i \geq 2$ be the smallest integer such that $c_{i} \in B_{j-1}.$ 
		We want to apply the bound \eqref{eq:corks2} $j$-many times. Notice that we can do this by \eqref{eq:corks1} together with $c_{i-i} \notin B_{j-1},$ which follows from the minimality of $i$.
		Then, applying the bound  \eqref{eq:corks2} repeatedly, we obtain 
		\[
		d(c_{i-1},z) \le (8+7A_1) d(c_{i},z) \le \cdots \le (8+7A_1)^{j}r_N
		\]
		so that $c_{i-1} \in B_{j}\setminus B_{j-1}.$ 
		In the case $j=1$ note that $c_N\in B_1\setminus B_0$.
		We have now shown that
		\begin{equation*}
			\mathcal{A}_j^{z} := \#  (\{c_i: 1 \le i \le N \} \cap (B_j  \setminus B_{j-1}) )
		\end{equation*}
	satisfies
		\begin{equation}
			\label{eq:corkproof2-no-empty-annuli} 
			\mathcal{A}_j^{z} > 0,\text{ for all $j \ge 1$ with $(8+7A_1)^{j} r_N \le d(c_1,z)$.}
		\end{equation}
		
		\textit{Step 2.} 
		We show that  a large portion of all the balls $B_j$ do not intersect the set $\{(x,t) : t \leq t_0\}$ and that the corresponding numbers $\mathcal{A}_j^z$ (with the same $j$) have a uniform upper bound.

        Since $P_N\subset \Omega_{t_0}^+$ and $A_3\widetilde{P}_N^3 \subset \Omega$ there holds that $r_N \leq 2^{-k+1}.$  Now if  $2^{-k/M}> 2^{-k+1}(8+7A_1)$ holds, which is equivalent with 
		\[
		k \geq   \frac{1}{1-1/M}\log_2(2(8+7A_1))
		\] 
		we may choose an integer $j_0\geq 1$ so that
		\begin{align}\label{eq:corks10}
			r_N (8+7A_1)^{j_0} \le 2^{-k/M} <  r_N (8+7A_1)^{j_0+1}.
		\end{align}
		By $r_N \leq 2^{-k+1}$ 
		and defining 
		\begin{align*}
			k_0 := 2\frac{\log_2(8+7A_1)-1}{1-1/M} > \frac{1}{1-1/M}\log_2(2(8+7A_1)),
		\end{align*} 
		a straightforward computation using the upper bound in \eqref{eq:corks10} shows that 
		\begin{align*}
			j_0 >
			\frac{k}{2\log_{2}(8+7A_1)},\qquad k \geq k_0.
		\end{align*}
		This is a good lower bound for $j_0.$
		For any integer $m \ge 1$,
		we write 
		\[
		j_0 = \# \{1 \le j \le j_0 : \mathcal{A}_j^{z} \ge m \} +\# \{1 \le j \le j_0 : \mathcal{A}_j^{z} < m \},
		\]
		where $\#$ denotes the cardinality of the set.
		By \eqref{eq:corkproof1} we have 
		\[
		m  \# \{1 \le j \le j_0 : \mathcal{A}_j^{z} \ge m \}  \le N \le 4 K k.  
		\]
		Define
		\[
		m_0 :=  4\log_{2}(8+7A_1) K.
		\]
		We have for $k\geq k_0$
        that
		\begin{align*}
			j_0 \leq  \frac{k}{4\log_{2}(8+7A_1)} +  \# \{1 \le j \le j_0 : \mathcal{A}_j^{z} < m_0 \}.
		\end{align*}
		Thus, after an absorption, we obtain 
		\begin{equation*}\label{eq:corkproof3:by-good-annuli}
			\frac{k}{4\log_{2}(8+7A_1)}  \leq \# \{1 \le j \le j_0 : \mathcal{A}_j^{z} < m_0 \},\qquad k \geq k_0.
		\end{equation*}
		Notice that $z\in L_k^+$ means that $\dist_d(z,\{t\le t_0\})>2^{-k/M}$ which together with the definition of $j_0$ automatically guarantees that $B_j\cap \{(x,t) : t \leq t_0\} = \emptyset$ for all $j=1,\dots, j_0.$ 
		Thus, in fact 
		\begin{equation}\label{eq:corks11}
			\begin{split}
				\frac{k}{4\log_{2}(8+7A_1)}  \leq \# \{1 \le j \le j_0 : B_j \cap \{t \le t_0\}=\emptyset \text{ and } \mathcal{A}_j^{z} < m_0 \}, \quad k \geq k_0.
			\end{split}
		\end{equation}
		
		\textit{Step 3.} We show that each annulus $B_{j}\setminus B_{j-1}$  with $0<\mathcal{A}_j^{z} < m_0$ admits a corkscrew.
		We have 
		\begin{multline*}
			\frac{1}{2}(8+7A_1)^j  r_{N}
			\leq \dist_d(B_{j-1},  B_{j}^{c})
			\le  14(1+A_1)\sum_{c_i \in B_{j}\setminus B_{j-1}} r_i < 14(1+A_1) (\max_{c_i \in B_{j}\setminus B_{j-1}}  r_i) m_0,
		\end{multline*}
		so there exists $c_{i_j}\in B_{j}\setminus B_{j-1}$ such that
		\begin{align}\label{eq:corks12}
			r_{i_j} \ge \frac{r_{N}(8+7A_1)^{j} }{28m_0(1+A_1)} =: \tilde{r}_j.
		\end{align}
		
		\textit{Step 4.} Now we are ready to construct the family $\mathcal{F}(k)$. Set 
		\[
		P_{z,j} = \begin{cases}
			B_d(c_{i_j},\tilde{r}_j) , \quad 0<\mathcal{A}_j^{z} < m_0, \\
			\varnothing, \quad {\rm otherwise}.
		\end{cases}
		\]
        Notice that if $j \leq j_0 - 1$ and $0<\mathcal{A}_j^{z} < m_0$, then by $c_{i_j}\in B_{j}$ and $B_{j+1} \subset \{t > t_0\}$ we have 
        \begin{align*}
            \dist_{d}(P_{z,j}, \{t\leq t_0\}) &\geq \dist_{d}(P_{z,j}, B_{j+1}^c) \geq r(B_{j+1}) - r(B_{j}) - \tilde{r}_j \\
            &= r_N(8+7A_1)^{j+1}\left( 1 - \frac{1}{(8+7A_1)} - \frac{1}{28m_0(1+A_1)(8+7A_1)}    \right) > 0.
        \end{align*}
        So there holds that $P_{z,j} \subset \Omega_{t_0}^+$ for each $j=1,\dots, j_0-1.$
        Now set 
        \begin{align*}
			D_1 := 28m_0(1+A_1), \qquad
			D_2 :=  1 + 4 \log_{2}(8+7A_1).
		\end{align*}
		By \eqref{eq:corks12} for $P_{z,j}$ with $0<\mathcal{A}_j^{z} < m_0$ there holds that $D_1P_{z,j}  \ni z.$
		Thus, by \eqref{eq:corks11}, it follows that
		\begin{multline*}
		    k \leq (D_{2}-1)\# \left( \{1 \le j \le j_0  : B_j \cap \{t=t_0\}=\emptyset \text{ and } \mathcal{A}_j^{z} < m_0 \} \right)
		  \\ 
        \leq D_{2} \# \left( \{1 \le j \le j_0 -1 : B_j \cap \{t=t_0\}=\emptyset \text{ and } \mathcal{A}_j^{z} < m_0 \} \right)
		\leq  \sum_{\substack{1 \le j \le j_0 -1 : \\ \mathcal{A}_j^{z} < m_0}} 1_{D_1 P_{z,j}}.
		\end{multline*}
		Applying the five covering lemma (Lemma 1.7 in \cite{MR2867756}) to the family
		\[
		\{ P_{z,j}  :   z \in L_k^+ , \ 0 \le j \le j_0 - 1\},
		\]
		we obtain a pairwise disjoint family $\mathcal{F}(k)$ with the desired properties. 
	\end{proof}
	
	Now we are in a position to conclude the proof of Lemma \ref{lemma:OLD}.
	
	\begin{proof}[Proof of Lemma \ref{lemma:OLD}]
		Recall that the non-centred Hardy--Littlewood maximal function
		\[
		Mf(z) = \sup_{ w \in \R^{n+1}, r > 0 }  \frac{1_{B_d(w,r)}(z)}{|B_d(w,r)|} \int_{B_d(w,r)} |f(z')| \, dz'
		\]
		satisfies 
		\begin{equation}
			\label{eq:hl-bound}
			\no{M}_{L^{q}(\R^{n+1}) \to L^{q}(\R^{n+1})} \leq Cq'
		\end{equation}
		for some absolute $C>0$,
		whenever $q \in (1,\infty]$ and $q' = q/(q-1)$ (and $\infty' = 1$).
		Thus,
		for any $D > 1$ and $q \in [1,\infty)$ and any disjoint family of parabolic cylinders $\mathcal{F}$,
		we have the well-known bound
		\begin{equation*}
			\label{eq:OLD-dilates}
			\begin{split}
				\Big\|\sum_{B \in \mathcal{F}} 1_{DB}\Big\|_{L^{q}(\R^{n+1})}\
				&= \sup_{\no{g}_{L^{q'}(\R^{n+1})\le 1}} \sum_{B \in \mathcal{F}} \int_{DB} g(z) \, dz \\ 
				&\le \sup_{\no{g}_{L^{q'}(\R^{n+1})\le 1}} \sum_{B \in \mathcal{F}} D^{n+p}|B| \inf_{z \in B} Mg(z) \\
				&\le D^{n+p} \Big\|\sum_{B \in \mathcal{F}} 1_{B}\Big\|_{L^{q}(\R^{n+1})} \sup_{\no{g}_{L^{q'}(\R^{n+1})\le 1}} \nos{Mg }_{L^{q'}(\R^{n+1})} \\
				& \leq CqD^{n+p} \Big\|\sum_{B \in \mathcal{F}} 1_{B}\Big\|_{L^{q}(\R^{n+1})} \leq CqD^{n+p}\Big|\bigcup \mathcal{F}\Big|^{1/q},
			\end{split}
		\end{equation*}
		where the penultimate bound used \eqref{eq:hl-bound} and the last the pairwise disjointness of the collection $\mathcal{F}.$
		For $a > 0$, we use the above to estimate
		\begin{equation*}
			\begin{split}
				\int_{\bigcup_{B \in \mathcal{F}} DB}  \exp\big(a \sum_{B \in \mathcal{F}} 1_{DB}(z) \big)  \, dz
				\le \sum_{j=0}^{\infty} \frac{a^{j}}{j!} \int  \Big( \sum_{B \in \mathcal{F}} 1_{DB}(z)  \Big)^{j} \, dz \lesssim  \Big|\bigcup \mathcal{F}\Big|   \sum_{j=0}^{\infty} \frac{(CD^{n+p} a j )^{j}}{j!} .
			\end{split}
		\end{equation*}
		Such a series converges provided that $a < 1/(CD^{n+p}e)$.
		
		Let $D_3 = D_1$ and $D_4 = D_2$, where  $D_1,D_2$ are as in the statement of Lemma \ref{lem:almost-bounded}.
		Applying the above bound with $\mathcal{F} = \mathcal{F}(k)$ and $D = D_1$ from Lemma \ref{lem:MeanCorkscrews} and $a = 1 / (2CD_1^{n+p}e)$
		we obtain
		\begin{align*}
			e^{a k/D_2}|L_k^+|  \le \int_{\bigcup_{B \in \mathcal{F}(k)} D_1B}  \exp\Big(a \sum_{B \in \mathcal{F}(k)} 1_{D_1B}(z) \Big)  \, dz 
			\lesssim \Bigl|\bigcup \mathcal{F}(k)\Bigr| \lesssim |\Omega_{t_0}^+ \cap B(c_{R_*},D_3k^{D_4})|.
		\end{align*}
		Rearranging gives 
		\begin{align*}
			|L_k^+| \lesssim   e^{-a k/D_2}|\Omega_{t_0}^+ \cap B(c_{R_*},D_3k^{D_4})|.
		\end{align*}
		Bounding the complement is more straightforward,
				\begin{align*}
			|L_k\setminus L_k^+| \leq 2^{-pk/M}|B_{\mathbb{R}^n}(z_0,D_3k^{D_4})|.
		\end{align*}
		Combining these estimates we obtain for some $\alpha>0$ that 
		\begin{align*}
			|L_k|  =	|L_k^+| + 	|L_k\setminus L_k^+| \lesssim \big(e^{-a k/D_2}+2^{-pk/M}\big)\big(D_3k^{D_4}\big)^{n+p} \lesssim 2^{-\alpha k},
		\end{align*}
		which finishes the proof of \eqref{eq:OLD:1}.
		Then we check \eqref{eq:OLD:2}.
		By Lemma \ref{lem:almost-bounded}  there holds for some $D_3>0$ that 
		\begin{align*}
			\Omega_{t_0}^+ \subset B(c_{R_*}, D_3) \cup \bigcup_{k \geq 1} L_k.
		\end{align*}
		Now, using \eqref{eq:OLD:1}, we bound
		\begin{align*}
			\big|\Omega_{t_0}^+\big| \lesssim D_3^{n+p} + \sum_{k\geq 1}2^{-\alpha k} \lesssim 1.
		\end{align*}
        This finishes the proof.
	\end{proof}

	\section{Proof of Lemma \ref{lemma:global-jn}: Level set estimate}
	\label{sec:global}
	
	We start the proof by an estimate that controls a $\pbmo^{+}$ function along a FIT chain.
	\begin{proposition}  
		\label{prop:basic-chain}
		Let $A = (A_1,A_2) \in [2,\infty)^{2}$ and $D \ge 1$.
		Let $C$ be a $\fit_{A}$ chain connecting the parabolic cylinder $P$ to the parabolic cylinder $Q$ and using the notation from the line  \eqref{gdef:lags} we denote 
		$$
		C_{P,Q} = \bigcup_{R \in C}  \widetilde{R}^{3},\qquad\widetilde{R}^{3} := \conv (P^{+,3} \cup P^{-,3}).
		$$ 
		For $R\in C$, 
		let $a_{R}$ be a constant with 
		\[
		\fint_{R^{+}} (f- a_{R})_{+} + \fint_{R^{-}} (f- a_{R})_{-} \le D \no{f}_{\op{PBMO}^+(C_{P,Q})}.
		\]
		Then, there holds that 
		\begin{align*}
			(a_Q - a_P)_+ \le A_2 D ( \# C )  \| f \|_{\op{PBMO}^+(C_{P,Q})}. 
		\end{align*}
	\end{proposition}
	\begin{proof}
		Denote $C = \{Q_{j}\}_{j=0}^{N}$.
		For $j \in \{2,\ldots,N\}$ and $V_j = Q_{j-1}^{+} \cap Q_{j}^{-}$,
		we see that
		\begin{multline*}
			(a_{Q_j} - a_{Q_{j-1}})_{+}
			\le (f_{V_j} - a_{Q_{j-1}})_{+} + (a_{Q_j} - f_{V_j})_{+}
			\le \fint_{V_j} (f - a_{Q_{j-1}})_{+} +  \fint_{V_j} (f - a_{Q_j})_{-} \\
			\le \frac{|Q_{{j-1}}^{+}|}{|V_j|} \fint_{Q_{{j-1}}^{+}} (f - a_{Q_{j-1}})_{+} 
			+ \frac{|Q_{j}^{-}|}{|V_j|} \fint_{Q_j^{-}} (f - a_{Q_j})_{-}
			\le  A_2 D \no{f}_{\op{PBMO}^+(C_{P,Q})},
		\end{multline*}
			where we used the estimate 
			\[
			\max\biggl( \frac{|Q_{{j-1}}^{+}|}{|V_j|},  \frac{|Q_{{j}}^{-}|}{|V_j|}\biggr) \leq  \frac{|Q_{j-1}\cup Q_{j}|}{|V_j|} \leq A_2.
			\]
		Thus,
		\[
		(a_Q - a_P)_+
		\le 
		\sum_{j=2}^{N}(a_{j} - a_{j-1})_{+} 
		\le A_2 D N \no{f}_{\op{PBMO}^+(C_{P,Q})}.
		\]
	\end{proof} 
	
	\begin{proof}[Proof of Lemma \ref{lemma:global-jn}]   
		Throughout proof $D$ denotes any constant with dependency on $D_1$ and $D_2$ in the statement and also on other absolute constants. 
		The same dependency is understood for the implicit constants.
		
		Let $R_{*}$ be the  central cylinder of a $(t_0,A)$-FIT H\"older domain (Definition \ref{defn:FITHolder}).
		For each $z \in \Omega_{t_0}^{+}$,
		let $B_z$ be a parabolic cylinder such that 
        $z\in B_z^+$ and $B_z\subset \Omega_{t_0}^+$ and
		\begin{equation}
			\label{eq:positive-direction-1}
			k_{\Omega,A}(R_{*},B_z) \le 2 K \log (1+1/\dist_d(z,(\Omega_{t_0}^+)^{c})),
		\end{equation}
	as is guaranteed by Definition \ref{defn:FITHolder}.
	
		By the metric five covering lemma (e.g. Lemma 1.7 in \cite{MR2867756})
		applied to the family $\{ B_z^+/ 5 : z \in \Omega^+_{t_0}  \}$,
		we obtain a countable family of pairwise disjoint metric balls $\{B_{z_i}^+/5:i \in \N\}$  so that letting $\mathcal{W} = \{B_{z_i}^+:i \in \N\}$ there holds that 
		\[ 
		\Omega_{t_0}^{+} = \bigcup_{i \in \N}  B_{z_i}^+
		\]
		and \eqref{eq:positive-direction-1} holds for every $B_{z_i}$.
		For $j \in \Z$,
		denote 
		\[
		\mathcal{W}_j = \{ B \in \mathcal{W}: 2^{j} \le r(B) < 2^{j+1} \},\qquad 
		\mathcal{W} = \bigcup_{j \in \Z} \mathcal{W}_{j},  \quad W_j = \bigcup_{B \in \mathcal{W}_j} B.
		\]  

			Fix $B^+\in \mathcal{W}$.
			By $B\subset \Omega_{t_0}^+$
            and $A_3\widetilde{B}^3 \subset \Omega$
            it follows that there exists some absolute constant $\sigma>1$ so that $\sigma B^+\subset\Omega_{t_0}^+$. Thus,  $\diam_d(B^+)\lesssim \dist_d(B^+,(\Omega_{t_0}^+)^c).$ 
			Hence for all $j\in\mathbb{Z}$ there holds that $W_j \subset \bigcup_{i=j-m}^{j+m}L_i$ for some absolute integer $m \geq 0.$ Thus, by the outer layer decay of $(t_0,A)$-FIT H\"older domains (Lemma \ref{lemma:OLD}),
			we have the estimate 
			\begin{equation}
				\label{eq:positive-direction-1-OLD}
				|W_j| \lesssim_{m,\alpha} 2^{\alpha j},
			\end{equation}
			for some $\alpha > 0$ only depending on the domain and the dimension.

		Consider $B^+ \in \mathcal{W}$ and 
		let $C_B = \{R_{B,k}\}_{k=1}^{\# C_B} \in \mathcal{C}_{\Omega,A}(R_{*},B)$ be a near optimal chain, that is
		\[
		2 k_{\Omega,A}(R_{*},B) \ge \# C_B . 
		\]  
		Let $f \in \op{PBMO}^{+}(\Omega_{t_0}^+)$. 
		Then by Theorem 3.1 in \cite{MR4657417},
		there exists $\varepsilon = \varepsilon(n,p) > 0$ such that for each $B \in \mathcal{W}$ and  $1\leq k \leq \# C_B$
		there exists a constant $a_{B,k}$ so that 
		\begin{multline}
			\label{eq:pos-dir-quote-JN}
			|\{ z \in R_{k}^{+}: (f(z)- a_{B,k} )_{+} > \lambda \}| \\
			+
			|\{ z \in R_{k}^{-}: (f(z)- a_{B,k} )_{-} > \lambda \}|
			\lesssim
			|R_k| \exp(- \varepsilon \lambda / \no{f}_{\op{PBMO}^+(3,1,A_3,R_k)} ),
		\end{multline} 
		for all $\lambda>0,$ and these same constants also satisfy
		\begin{equation}
			\fint_{R_{k}^{+}} (f(z)- a_{B,k})_{+}\,dz + \fint_{R_{k}^{-}} (f(z)- a_{B,k})_{-}\,dz \lesssim \no{f}_{\op{PBMO}^+(3,1,A_3,R_k)}.
		\end{equation} 
		Let $a_{R_*} := a_{B,1}$ and notice that this does not depend on $B \in \mathcal{W},$ since all the chains $C_B$ begin from $R_*.$
		Now, write 
		\begin{equation}
			\label{eq:JN-level-sets-LHS}
			|\{ z \in \Omega_{t_0}^{+}: (f(z)-a_{R_{*}})_{+} > \lambda \}|
			\le \sum_{j \in \Z} \sum_{B^+ \in \mathcal{W}_j} 
			|\{ z \in B^+: (f(z)-a_{B,1})_{+} > \lambda \}| .
		\end{equation}
		Then,
		\begin{multline*}
			|\{ z \in B^+: (f(z)-a_{B,1})_+ > \lambda \}|
			\le |\{ z \in B^+: (f(z)- a_{B, \# C_B} )_{+} > \lambda/2 \}| \\
			+ 
			|\{ z \in B^+: (a_{B,\# C_B} - a_{B,1} )_{+} > \lambda/2 \}| 
			=: \I_B + \II_B.
		\end{multline*} 
		By \eqref{eq:pos-dir-quote-JN}  we see that 
		\[
		\I_B \lesssim  |B^+|\exp(- \varepsilon \lambda / \no{f}_{\op{PBMO}^+(3,1,A_3,B)} )  \leq  |B^+|\exp(- \varepsilon \lambda / \no{f}_{\op{PBMO}^+(3,1,A_3,\Omega)}).
		\] 
		Since the sets $B^+/5$ are disjoint and contained in $\Omega_{t_0}^+$, this leads to the bound
		\begin{align*}
			\sum_{j \in \Z} \sum_{B \in \mathcal{W}_j}  \I_B\lesssim 
			D |\Omega_{t_0}^+ | \exp(- \varepsilon \lambda / \no{f}_{\op{PBMO}^+(3,1,A_3,\Omega)}),
		\end{align*}
		which is of the correct form.

			Then we estimate the sum $ \sum_{j \in \Z} \sum_{B \in \mathcal{W}_j}  \II_B$.  Recall that in the following estimates $D$ stands for any absolute constant that may change from line to line. 
			Without loss of generality we may assume that $\no{f}_{\op{PBMO}^+(3,1,A_3,\Omega)} = 1.$
			Fix a cylinder top $B^+ \in\mathcal{W}_j.$  By $\diam_d(B^+)\lesssim \dist_d(B^+,(\Omega_{t_0}^+)^c)$ it holds that  $\dist_d(z,(\Omega_{t_0}^+)^c) \geq D2^{j}$
            for $z\in B^+$.
			By Proposition \ref{prop:basic-chain} and the choice of $C_B$ and \eqref{eq:positive-direction-1} ($(t_0,A)$-FIT H\"older domain)  we bound 
			\begin{align*}
				(a_{B,\# C_B} - a_{B,1} )^{+} &\leq A_2D \# C_B \leq  A_2D  k_{\Omega,A}(R_{*},B) \leq A_2D  K \log (1+(A_3D)^{-1}2^{-j}).
			\end{align*}
			Hence, if $\lambda/2 \geq A_2D 8 K\log (1+(A_3D)^{-1}2^{-j})$, then $\II_B = 0;$ or equivalently $\II_B\not = 0$ only if 
			\[
			j \leq -\frac{\lambda}{A_2D K} - \log_2(1/(A_3D)) =: j_0.
			\] 
			Thus,
			\begin{align*}
				\sum_{j \in \Z} \sum_{B^+ \in \mathcal{W}_j}  \II_B &\leq 
				\sum_{j \leq j_0} \sum_{B^+ \in \mathcal{W}_j} |B| \lesssim  \sum_{j \leq j_0} |W_j| \lesssim \sum_{j \leq j_0} 2^{j\alpha} \\ 
				&\lesssim DA_3 \exp \Bigl(  - \frac{\alpha \lambda}{A_2 D K} 
				\Bigr)= DA_3 \exp \biggl(  - \frac{\alpha \lambda}{A_2 D K\no{f}_{\op{PBMO}^+(3,1,A_3,\Omega)}} 
				\biggr).
			\end{align*}

		Combining the bounds, we have now shown that 
		\begin{multline*}
			|\{ z \in \Omega_{t_0}^{+}: (f(z)-a_{R_{*}})_{+} > \lambda \}| \\ \lesssim  \bigl(|\Omega_{t_0}^+ | + 1\bigr)\exp\biggl(-\min\Bigl(\varepsilon, \frac{\alpha}{A_2 D K}\Bigr) \frac{\lambda}{ \no{f}_{\op{PBMO}^+(3,1,A_3,\Omega)}}\biggr).
		\end{multline*}
			We are done.
	\end{proof}

	\section{The necessity of exponential integrability}
	\label{sec:universal} 
	
	We start the proof
	by the following heuristic statement.
	Consider an interior parabolic cylinder properly nested inside another dubbed the outer cylinder.
	The length of any optimal FIT chain beginning from the future part of the interior cylinder, can be controlled by the length of any chain beginning from the past part of the interior cylinder, provided they both terminate in the same reference rectangle in the outer cylinder.
	The absolute \emph{additive} constant in this bound depends only on the reference rectangle in question.
	This will follow by compactness,
	but the proof requires a way to produce lower bounds for the lengths of optimal chains for which we resort to the PDE estimates from Lemma \ref{lemma:Barenblatt}.

	\begin{lemma}
		\label{lem:compactness-contradiction-chain}
		Let $A \in [2,\infty)^{2}$.
		There exists $\delta_0 = \delta_0(n,p,A)  > 0$ such that for every $\delta \in (0,\delta_0)$
		the following holds. There exists a constant $c = c(\delta,n,p,A)>0$ such that if $R_0$ is a parabolic cylinder with
		$z^{+} \in R_{0}^{+}$, $z^{-} \in R_0^{-}$ and $O_0$ is a parabolic cylinder with $O_0 \cap \partial (\delta^{-1} R_0)\ne \varnothing$, then  there holds that 
		\begin{equation}
			\label{eq:prop.cc.statement}
			\begin{split}
				&\inf \{ \# C:  \rho > 0,\  C \in \mathcal{C}_{\R^{n+1},A}(O_0, B_d(z^{+},\rho) )\} \\
				&\le \inf \{ \# C:  \rho > 0,\  C \in \mathcal{C}_{\R^{n+1},A}(O_0, B_d(z^{-},\rho) )\} 
				+ c.
			\end{split}
		\end{equation}
	\end{lemma}
	\begin{proof}
		By parabolic scaling and translation,
		we may assume that $R_{0} = B_d(0,1)$.
		Further,
		inside this proof,
		we abbreviate $\mathcal{C}(P,Q) := \mathcal{C}_{\R^{n+1},A}(P,Q)$
		and refer to $\fit_{A}$ chains as simply FIT chains.
		
		If the right-hand side of \eqref{eq:prop.cc.statement} is infinite,
		the claim holds trivially 
		so we may assume that it is finite.
		Assume then, for a contradiction, that the claim does not hold,
		that is, for every $i \in \N$ there exist
		$z_i^{+} \in R_{0}^{+}$, $z_i^{-} \in R_0^{-}$ and a parabolic cylinder $O_i$ with $O_i \cap \partial (\delta^{-1} R_0 )\ne \varnothing$ such that 
		\begin{equation}
			\label{eq:com-con-proof-main}
			\inf \{ \# C: \ \rho > 0,\  C \in \mathcal{C}(O_i, B_d(z_i^{+},\rho))\} 
			> \inf \{ \# C: \ \rho > 0,\  C \in \mathcal{C}(O_i, B_d(z_i^{-},\rho))\} + i .
		\end{equation}
		By compactness,
		we may assume that the sequence $(z_i^{+},z_{i}^{-})$ converges to a pair of points $(z^{+},z^{-})$ in the closure of $R_0^{+} \times R_0^{-}$.
		Since 
		\begin{multline}
			\inf \{ \# C: \ \rho\geq\delta /10, \ C \in \mathcal{C}(O_i, B_d(z_i^{+},\rho))\}  \\
			\ge \inf \{ \# C: \ \rho > 0,\  C \in \mathcal{C}(O_i, B_d(z_i^{+},\rho))\} 
			>  i \rightarrow \infty,
		\end{multline}
	we conclude that $|O_i| \rightarrow 0$.
		By compactness,
		we may assume that there exists a point $z_O \in \partial (\delta^{-1} R_0)$ such that 
		$\dist(z_O, O_i) \rightarrow 0$
		as $i \to \infty$.
		Next we examine separately the cases when $\pi_T(z_O) < \lim_{i} \pi_T(z_i^-)$ and $\pi_T(z_O) = \lim_{i} \pi_T(z_i^-)$
		and show that in each a contradiction follows.

		\subsubsection*{The case $\pi_T(z_O) < \lim_{i} \pi_T(z_i^-)$}
		We assume equivalently that
		$\pi_T(z_O) = \lim_{i} \pi_T(z_i^{-}) - \eta$ for some $\eta > 0$.
		Passing to a subsequence,
		we may assume $\pi_T(z_i^{-}) >  \pi_T(z_O) + \eta / 2$ for all $i$. 
    Given an arbitrary $\rho = \rho(\eta)>0$ sufficiently small and $C \in \mathcal{C}(O_i, B_d(z_i^-,\rho))$, we next locate a large cylinder (large, uniformly in $i$) in the chain $C$. That such a cylinder exists 
     follows from
     $\eta>0$ and $r(O_i) \to 0$ as $i\to\infty$. We provide the details next.

		Now,
		for small enough $\tilde\rho = \tilde\rho(\eta,z_O)>0$,
		a radius $\rho \in (0,\tilde{\rho})$ and an intermediate radius $\rho' \in (\rho, \tilde{\rho})$,
		we can find for every large enough $i = i(\rho)$ (so that $r(O_i) \lesssim \rho'$) a chain $C_{\rho,i} \in \mathcal{C}(O_i, B_d(z_i^{-},\rho))$ with
		\[
		\# C_{\rho,i} \lesssim  \log\frac{\rho'}{\rho} + \frac{\dist_{n}(\pi_S(B_d(z_i^{-},\rho)),\pi_S(O_i))}{\rho'} + \frac{1}{(\rho')^{p}} + \log \frac{\rho'}{r(O_i)} =: f_i(\rho,\rho').
		\]
		This chain can be constructed by scaling the size of rectangles from $\rho$ to $\rho'$,
		moving laterally to match the space coordinate,
		moving vertically to match the time coordinate,
		and by scaling from $\rho'$ to $r(O_i)$.

		Keeping the parameters $0 < \rho < \rho' < \tilde{\rho}$ fixed for a moment, 
		for a chain $C$ in the family $\mathcal{C}(O_i, B_d(z_i^{-},\rho))$ with no rectangles larger than $\rho'$,
		we observe that
		up to a multiplicative constant
		$f_i(\rho,\rho')$ is also a lower bound for $\# C$.
		Therefore, for a chain $C \in \mathcal{C}(O_i, B_d(z_i^{-},\rho))$ with no rectangles larger than $\rho'$, we have $\# C \sim f_i(\rho,\rho')$.
		We observe that $f_i$ is decreasing in both variables separately in a neighborhood of $0$
		and 
		\[
		\lim_{\rho \to 0} \inf_{\rho < \rho' < \tilde{\rho}} f_i(\rho,\rho') = \infty,
		\]
        and that the minimum point of $f_i(\rho,\rho')$ does not depend on $i$.
		Thus, it follows that
		there exists $\rho_1 > 0$ independent of $i$ (large enough) such that for 
		every pair $(\rho,C)$ almost minimizing $\# C$ over $\{(\rho, C): \rho > 0, C \in \mathcal{C}(O_i, B_d(z_i^{-},\rho)) \}$
		\begin{itemize}
			\item either $\rho > \rho_1$ 
			\item or there exists $P \in C$ with $r(P) > \tilde{\rho}$.
		\end{itemize}
		Summarising, for all $i$ sufficiently large,
		for an almost optimal $C \in \mathcal{C}(O_i, B_d(z_i^-,\rho)),$ there exists some $\rho_{2,i} > \rho_2 := \min( \tilde{\rho}, \rho_1) $ and a point $\zeta_i \in 2\delta^{-1}R_0$ (depending on which of the above two alternatives is realised)
		with $\pi_T(\zeta_i) \leq \pi_T(z_i^{-})$ 
		such that $B_d(\zeta_i,\rho_{2,i}) \in C$. 
		
		To conclude, it suffices to note that any chain in $\mathcal{C}(O_i, B_d(z_i^{+},\delta /10))$ can be split in two parts: 
		one of them connecting $B_d(\zeta_i,\rho_{2,i})$ to $B_d(z_i^{+},\delta /10)$ with $c(\rho_2,\delta)$ many parabolic cylinders
		and the other part connecting $O_i$ to $B_d(\zeta_i,\rho_{2,i})$.
		Hence 
		\begin{multline*}
			i \le \inf \{ \# C: \ C \in \mathcal{C}(O_i, B_d(z_i^{+},\delta /10))\} - \inf \{ \# C: \ \rho > 0,\  C \in \mathcal{C}(O_i, B_d(z_i^{-},\rho))\}  \\ 
			\le \sup_{i} \inf \{ \# C:   C \in \mathcal{C}( B_d(\zeta_{i},\rho_{2,i}),  B_d(z_i^{+},\delta /10))\}  
			\le c (\rho_{2},\delta),
		\end{multline*}
		which is a contradiction.
		Thus, it cannot hold that $\eta > 0$.

		\subsubsection*{The case $\pi_T(z_O) = \lim_{i} \pi_T(z_i^-)$} 
		We may always build a trivial chain $C^{+} \in\mathcal{C}(O_i, B_d(z_i^{+},\delta /10))$
		by first using a chain consisting of parabolic cylinders of radius $\sim\delta/10$ to match the space coordinates, 
		then adjusting the time coordinate 
		and finally adjusting the scale coordinate in an optimal manner.
		For some absolute constant $C>0$, such a chain obeys the estimate 
		\begin{equation}
			\label{17abril-1}
			\inf \{ \# C: \ C \in \mathcal{C}(O_i, B_d(z_i^{+},\delta /10))\} 
			\le   C +  S(1,r(O_i)),
		\end{equation}  
		where $S(\rho,r)$, with $\rho,r > 0$, is the minimal number of rectangles that a FIT chain starting from an $r$-rectangle and ending with a $\rho$-rectangle can have, and where we also used that $r(O_i) \lesssim \delta.$
		 Clearly we have the estimate 
		\begin{equation}
			\label{eq:10juliol1645}
			S(r,\rho) \sim \log \frac{r}{\rho}.
		\end{equation}

We divide the treatment into two subcases. 
For each $i$,
denote 
$$
h_i := \dist_1(\pi_T(z_i^{-}), \pi_T(O_i)).
$$ 

\textit{Subcase 1.}
Assume first that for infinitely many $i$,
it holds that $h_i \leq Dr(O_i)^{p}$ for some absolute constant $D\geq100$. 
By restricting to a subsequence,
we may assume this inequality for all $i$.
Then, Lemma \ref{lemma:Barenblatt} together with the assumption of the present subcase allows us
to estimate
\begin{align*}
	\inf \{ \# C: \ \rho > 0,\  C \in \mathcal{C}(O_i, B_d(z_i^{-},\rho))\} &\gtrsim \biggl(\frac{|\pi_{S}(z_i^-)-\pi_{S}(c_{O_i})|^p}{h_i + 2r(O_i)^{p} } \biggr)^{\frac{1}{p-1}}  \\
	&\gtrsim_{\delta} \biggl(\frac{1}{r(O_i)^{p} } \biggr)^{\frac{1}{p-1}},
\end{align*}
where we also used that $R_0 = B_d(0,1)$. 
Now using this and the bounds \eqref{eq:com-con-proof-main} and \eqref{17abril-1} we obtain for some absolute constants $c_1,c_2>0$ that 
\begin{multline*}
	i \le \inf \{ \# C: \ C \in \mathcal{C}(O_i, B_d(z_i^{+},\delta /10))\}  - \inf \{ \# C: \ \rho > 0,\  C \in \mathcal{C}(O_i, B_d(z_i^{-},\rho))\} \\
	\le c_1  + c_1 \log \frac{1}{r(O_i)} -  c_2 \biggl(\frac{1}{r(O_i)^{p} } \biggr)^{\frac{1}{p-1}} \longrightarrow - \infty
\end{multline*}
as $i \to \infty$, provided that $h_i \le D r(O_i)^{p}$. This is a contradiction and hence $h_i \le D r(O_i)^{p}$
cannot hold for infinitely many $i$.
		
\textit{Subcase 2.}
		We can now assume that there is a sequence of indices $i$ such that $h_i > Dr(O_i)^{p}$ for all $i.$
		We first use the chain $C^{+}\in \mathcal{C}(O_i, B_d(z_i^{+},\delta /10))$ as above, matching the space coordinate, almost matching the time coordinate and matching the scale.
		This leads to the bound, where we track the transition through the scale $h_i^{1/p}+r(O_i)$,
		\[
		\# C^{+} \le C + S(1,r(O_i))\leq  C + S(1,h_i^{1/p}+r(O_i)) + S(h_i^{1/p}+r(O_i), r(O_i)).
		\]
		Then fix $\rho >0$ and choose a chain $C^{-}$ that is almost minimally long in $\mathcal{C}(O_i, B_d(z_i^{-},\rho))$.
		Denote by $C^{--}$ a minimal collection of $S(h_i^{1/p}+r(O_i),r(O_i))$ rectangles in $C^{-}$
		that are needed to transition from the scale $r(O_i)$ to the scale $h_i^{1/p}+r(O_i)$.
	Since $h_i > Dr(O_i)^{p} \gg r(O_i)^{p}$, we can assume that $C^{--}$ is a strictly increasing chain (with respect to radius) where the volume ratio between the subsequent cylinders in the chain is $\sim A_1$, and this chain begins from the cylinder $O_i$ and terminates before the cylinder $B_d(z_i^-,\rho).$ Moreover, by this choice it is immediate that for some absolute constant $c_1$ there holds that 
		\[
		\# C^{--} = c_1 + S(r(O_i), r(O_i) + h_i^{1/p}).
		\]
	Now we estimate in the spatial variable 
	\[
	\diam_{S}\Bigl(  \bigcup_{P\in C^{-+}}P   \Bigr) 
	\gtrsim 1 - \sum_{P \in C^{--}} \diam_{S}(P)
	\gtrsim_{A_1} 1 - h_i^{1/p} \gtrsim 1.
	\] 
	Hence, $C^{-+}$ is a FIT chain that starts from scale $h_i^{1/p}+r(O_i)$
	and ends at scale $\sim \rho$,
	covers a spatial distance $\sim 1$, 
	and covers a temporal distance $\sim h_i$. 
    By Lemma \ref{lemma:Barenblatt}, we have
		\[
		\# C^{-+} \gtrsim \biggl(\frac{1}{h_i + r(O_i)^{p} } \biggr)^{\frac{1}{p-1}} .
		\]
		Altogether, recalling \eqref{eq:10juliol1645}, we obtain 
		\begin{multline*}
			i \le c_2 +  \# C^{+} - \# C^{-}  = c_2 + \# C^{+} - \# C^{-+} - \# C^{--} \\
			\le c_2 + \Bigl(S(1,h_i^{1/p}+r(O_i)) + S(h_i^{1/p}+r(O_i), r(O_i)) \Bigr) \\
			-   S(h_i^{1/p}+r(O_i),r(O_i)) - c_3\biggl(\frac{1}{h_i + r(O_i)^{p} } \biggr)^{\frac{1}{p-1}}    \\
			\lesssim c_2-c_1 + \log \frac{1}{h_i + r(O_i)^{p}} - c_3 \biggl(\frac{1}{h_i + r(O_i)^{p} } \biggr)^{\frac{1}{p-1}}  
			\longrightarrow -\infty ,
		\end{multline*}
		which is a contradiction.
		This completes the proof.
	\end{proof}

Passing the conclusion to a parabolic cylinder well contained in a domain is the content of the following Lemma \ref{lem:universal-pbmo}.
This is actually enough to show that the counting function must be in a parabolic BMO class of Definition \ref{def:PBMO} and by Corollary \ref{cor:equivalences} in all of them.

\begin{lemma}
\label{lem:universal-pbmo}
Let $t_0 \in \R,$ $A \in [2,\infty)^{3}$ and let $\Omega$ be a $t_0$-FIT connected domain with a central cylinder $R_{*}$.
Let $\delta= \delta(n,p,A)$ be as in Lemma \ref{lem:compactness-contradiction-chain}.
Define 
\[
k(z) := \inf_{0 < \rho < 10^{-2}\delta \dist_{d} (z, (\Omega_{t_0}^{+})^c) } k_{\Omega,A}(R_{*},B_d(z,\rho)).
\]
Let $R$ be a cylinder so that $10 A_3 \delta^{-1} R \subset \Omega_{t_0}^{+}$.
Then there exists $c = c(\delta,n,p,A)>0$ such that 
\begin{equation}\label{eq:univ-bmo-statement}
\sup_{z^{+} \in R^{+}} \sup_{z^{-} \in R^{-}} ( k(z^{+}) - k(z^{-}) )_{+} \le c,
\end{equation}
and further 
\begin{equation}
	\label{eq:univ-bmo-statement2}
	\| k \|_{{\pbmo}^{+}(6,1,20\delta^{-1},\Omega_{t_0}^{+})} \le 2c.
\end{equation}
\end{lemma}
\begin{proof} 
We first prove \eqref{eq:univ-bmo-statement}.
Without loss of generality, we may assume that $\diam_d(R) = 1$
and that all the chains provided by the $t_0$-FIT connectivity are in the classes $\mathcal{C}_{\Omega,A}(R_{*}, \cdot)$ (Proposition \ref{prop:varying-chains}).
Pick $z^{\pm} \in R^{\pm}$. 
Let $\rho^- \in (0, 10^{-2}\delta \dist_{d} (z^{-}, (\Omega_{t_0}^{+})^c )$.
Consider a FIT chain $C^- \in \mathcal{C}_{\Omega,A}(R_{*}, B_d(z^{-},\rho^{-})).$
Write $C^- = \{P_j\}_{j=1}^{N},$ where $P_1 = R_*$ and $P_N = B_d(z^{-},\rho^{-})$.
Let $i \in \{2,\ldots  N-1\}$ be the smallest integer such that $P_i \cap\delta^{-1} R\ne \varnothing$ and denote $O := P_i$. 
It follows $r(O) \le\delta^{-1}$ so that from $10A_3\delta^{-1} R \subset \Omega_{t_0}^+$ we conclude that 
\[
\# C^- \ge k_{\Omega,A}(R_{*},O) + \inf \{ \# C:  \rho > 0,\  C \in \mathcal{C}_{\R^{n+1},A}(O, B_d(z^{-},\rho) )\}.
\]
Taking infimum over all $C^- \in \mathcal{C}_{\Omega,A}(R_{*}, B_d(z^{-},\rho^{-}))$ and 
$\rho^-$,
and by Lemma \ref{lem:compactness-contradiction-chain}, we estimate
\begin{multline*}
k(z^{+}) - k(z^{-}) \le   
\bigl[ D \log\delta^{-1} +
\inf \{ \# C:  \rho > 0,\  C \in \mathcal{C}_{\R^{n+1},A}(O, B_d(z^{+},\rho) )\} + k_{\Omega,A}(R_{*},O)  \bigr] \\
- \bigl[k_{\Omega,A}(R_{*},O) + \inf \{ \# C:  \rho > 0,\  C \in \mathcal{C}_{\R^{n+1},A}(O, B_d(z^{-},\rho) )\} \bigr]
\le c(\delta,n,p,A),
\end{multline*}
where the term $D\log\delta^{-1}$ comes from the difference between lengths of chains from $O$ to $B_d(z^+,\rho)$ with 
constrained and free $\rho>0$.
The bound \eqref{eq:univ-bmo-statement} follows.

For the bound \eqref{eq:univ-bmo-statement2} we consider all cylinders $P$ such that
$P^{\pm,3}$ satisfy the hypothesis for $R$ in the first part of the statement, that is, $10A_3\delta^{-1} P^{\pm,3} \subset \Omega^+_{t_0}$.
Taking any $z' \in P$ and using \eqref{eq:univ-bmo-statement} , we estimate 
\begin{multline*}
	\inf_{c \in \R} \Big( \fint_{P^{+,6}}(k(z)-c)_{+}  \, dz+ \fint_{P^{-,6}}(k(z)-c)_{-} \, dz \Big) \\ \le    \Big( \fint_{P^{+,6}}(k(z)-k(z'))_{+}  \, dz + \fint_{P^{-,6}}(k(z)-k(z'))_{-} \, dz \Big) \\
	\le  \sup_{z^{+} \in P^{+,6}} \sup_{z^{-} \in P} ( k(z^{+}) - k(z^{-}) )_{+}  +\sup_{z^{+} \in P} \sup_{z^{-} \in P^{-,6}} ( k(z^{+}) - k(z^{-}) )_{+}  \leq 2c.
\end{multline*}
The claim now follows by taking the supremum over $P$.
\end{proof}
 
The final step of the proof is up next.
Unlike in the setting of \cite{SmiSte91b},
we have to deal with the time-directedness of chains in the definition of the function $k_{\Omega,A}$ from Lemma \ref{lem:universal-pbmo}.

\begin{lemma}\label{lem:necessity}
	Let $t_0 \in \R$, $\varepsilon > 0,$ $A = (A_1,A_2,A_3) \in[2,\infty)^{3}$ and $\Omega \subset \R^{n+1}$ be fixed. Suppose that
	\begin{itemize}
		\item 	$\Omega$ is a $(t_0-\varepsilon)$-FIT connected with a central cylinder $R_* \subset \Omega_{t_0-\varepsilon}^- $,
		\item 
        there exist $D_0,\eta > 0$ so that all $f\in \pbmo^{+}(\Omega_{t_0-\varepsilon}^+)$ satisfy
		\begin{align}\label{eq:lem:necessity}
			\inf_{a \in \R} \int_{\Omega_{t_0}^{+}} \exp \Bigl( \eta (f(z) - a)_{+} / \no{f}_{\pbmo^{+}(\Omega_{t_0-\varepsilon}^+)} \Bigr) \, dz 
			\le D_0
		\end{align}
	\end{itemize}

	Then, $\Omega$ is a $(t_0,A)$-FIT H\"older domain with the central cylinder $R_*$ and there exists a constant $D=D(A,n,p,D_0,\eta,\Omega^+_{t_0})>0$ such that
	\begin{equation}\label{eq:lem:necessity2}
        \inf_{ z\in Q^+, Q\subset\Omega^+_{t_0}}
        k_{\Omega,A}(R_{*},Q) 
		\leq 
        D
        \log\biggl(\frac{1}{\dist_d(z,(\Omega^+_{t_0})^c)} +1 \biggr),\qquad z\in\Omega_{t_0}^+.
	\end{equation}
\end{lemma}

\begin{proof}
	Since $\Omega$ is a $(t_0-\varepsilon)$-FIT connected with central cylinder $R_* \subset \Omega_{t_0-\varepsilon}^- $, Lemma \ref{lem:universal-pbmo} implies that
	the function 
	\[
	k(z) := \inf_{0 < \rho < 10^{-2}\delta \dist_{d} (z, (\Omega_{t_0-\varepsilon}^{+})^c)  } k_{\Omega,A}(R_*,B_d(z,\rho))
	\]	
	satisfies 
	\[
	k\in {\pbmo}^{+}(6,1,20\delta^{-1},\Omega_{t_0-\varepsilon}^{+}).
	\]
	By Corollary \ref{cor:equivalences} we further deduce that 
$$
k \in {\pbmo}^{+}(3,1,1,\Omega_{t_0-\varepsilon}^{+}) := \pbmo^{+}(\Omega_{t_0-\varepsilon}^{+}).
$$ 
    By the hypothesis \eqref{eq:lem:necessity} there exists an absolute constant $a\in\mathbb{R}$
 such that for all $\lambda>0$ there holds
	\begin{multline*}
	 \exp(\eta \lambda / D_1 ) 
     \abs{ \{ z \in \Omega_{t_0}^{+}: (	k(z) - a )_+ > \lambda \} }	\\
		 \leq \int_{\Omega_{t_0}^{+}} \exp \Bigl( \eta (k(z) - a)_{+} / \no{k}_{\pbmo^{+}(\Omega_{t_0-\varepsilon}^+)} \Bigr) \, dz  
         \leq 
		 D_0 .
	\end{multline*}
    Fix $z_0\in \Omega^+_{t_0}$.
	Let 
    $R$
    be a parabolic cylinder such that $c(R^+)=z_0$ and $10A_3\delta^{-1} R \subset \Omega_{t_0}^{+}\subset \Omega_{t_0-\varepsilon}^{+}$, that is, satisfying the hypothesis of Lemma \ref{lem:universal-pbmo}, with
    \begin{align*}
	  \dist_d(z_0, (\Omega_{t_0}^{+})^c) = D_2(\delta,A_3) |R|^{1/(n+p)}.
	\end{align*}
	By \eqref{eq:univ-bmo-statement}
	there exists $c= c(n,p,A)$ so that 
	\begin{align*}
		|R| &\leq |\{z \in R^{-} : k(z) > \sup_{z^+ \in R^{+}} k(z^+) - c  \}| \\
		&\leq  |\{z \in \Omega_{t_0}^{+} : (k(z) - a)_+ > \sup_{z^+ \in R^{+}} k(z^+) - c - a  \}| \\
		&\leq D_0 \exp\Bigl(-\eta \bigl( \sup_{z^+ \in R^{+}} k(z^+) -c - a\bigr)/D_1 \Bigr),
	\end{align*}
    whenever $\lambda_0:=\bigl( \sup_{z^+ \in R^{+}} k(z^+) -c - a\bigr) > 1$. If $\lambda_0 \leq 1$, then by a simple estimate
	\begin{multline*}
		|R| \leq |\Omega_{t_0}^{+}|\leq e^{\eta/D_1} e^{-\eta/D_1} \int_{\Omega_{t_0}^{+}} \exp \Bigl( \eta (k(z) - a)_{+} / \no{k}_{\pbmo^{+}(\Omega_{t_0-\varepsilon}^+)} \Bigr) \, dz   \\
        \leq
        e^{\eta/D_1} D_0 \exp\Bigl(-\eta \bigl( \sup_{z^+ \in R^{+}} k(z^+) -c - a\bigr) /D_1 \Bigr).
	\end{multline*}
	Thus, using also that $|\Omega_{t_0}^+|<\infty$ (hence cannot contain points far away from the boundary),
    we have
	\begin{multline*}
			k(z_0)\leq
            \sup_{z^+ \in R^{+}} k(z^+) \leq  \frac{D_1}{\eta}\log \biggl( \frac{	D_0  D_2^{n+p}}{\dist_d(z_0,(\Omega_{t_0}^{+})^c)^{n+p}}\biggr) + \frac{\eta}{D_1} +c + a \\
			\leq 
            D(A,n,p,D_0,\eta,\Omega^+_{t_0})
        \log\biggl(\frac{1}{\dist(z_0,(\Omega_{t_0}^{+})^c)} +1 \biggr).
	\end{multline*}
Since this holds for all $z_0\in \Omega^+_{t_0}$ and
the left-hand side dominates the left-hand side of \eqref{eq:lem:necessity2},
the claim follows.
We are done.
\end{proof}

\section{Examples of domains}
\label{sec:examples}
\subsection*{Review of basics in the stationary setting}
We conclude our study by informally discussing how $t_0$-FIT connectivity and $(t_0,A)$-H\"older conditions compare to other classes of domains. 

We begin by discussing John domains which form a subset of the domains satisfying the quasihyperbolic boundary condition (this inclusion is a rather immediate consequence of definitions), also known as the H\"older condition.
A domain $\Omega \subset \R^{n}$ is a John domain \cite{MR565886}
if there exists a central point $z_* \in \Omega$ and a constant $C_{\Omega}>0$ such that for all $z \in \Omega$ there exists an arc length parametrized rectifiable curve $\gamma$ such that  $\gamma(0) = z$, $\gamma(\ell(\gamma)) = z_*$ and for all $s \in (0,\ell(\gamma))$ there holds that 
\[
\int_{0}^{s} \, d \ell \ge C_{\Omega} \dist(\gamma(s),\Omega^{c}). 
\]
Under mild connectivity assumptions John domains are known \cite{BuKoGu96} to be equivalent with Boman chain sets \cite{Boman82}. 
John domains are also known as those with the twisted cone condition, prototypical examples are domains with Lipschitz regular boundaries or the von Koch snowflake.

In the plane, a simply connected domain whose Riemann mapping has a global quasiconformal extension
is known to be uniform and hence John \cite{MR565886}.
On the other hand, every global quasiconformal mapping is locally H\"older continuous (see for instance Corollary 3.10.3 in \cite{MR2472875}).
Now, the class of simply connected domains whose Riemann mapping extends H\"older continuously but not necessarily quasiconformally is the class of H\"older domains \cite{BecPom82}.
This is one way to see the difference between the John condition and the quasihyperbolic boundary condition.

Despite the name, 
bounded domains with $\alpha$-H\"older boundaries, for $\alpha \in (0,1)$, 
but with an exterior cusp fail the quasihyperbolic boundary condition or equivalently the H-chain condition \cite{Buck99} (which inspires the $(t_0,A)$-FIT H\"older condition) and hence they also fail the John condition (as John condition implies quasihyperbolic boundary condition). 
In \cite{MR889369}, an explicit example of a H\"older domain that is not a John domain is constructed.

Finally, the boundary behavior of functions is tightly connected to their extendability from the domain to the full ambient space. In the case of BMO, the class of bounded domains in which a bounded extension operator exists was characterized by Jones \cite{Jon80Ext} and later shown to coincide with bounded uniform domains \cite{GehOsg79}, a proper subclass of John domains.
Hence, the class of H\"older domains is considerably larger than the class of bounded extension domains.

\subsection*{Example 1: Cylinders and graphical domains are $(t_0,A)$-H\"older}
In case $\Omega \subset \R^{n}$ is a H\"older domain,
then by \cite{Saa16} the domain $\Omega \times (T_1,T_2)$ with $T_1 < T_2$ is a $(t_0,A)$-FIT H\"older domain for all $t_0 \in (T_1,T_2)$.
Similarly, in case $\psi : \R^{n+1} \to (0,\infty)$ is Lipschitz continuous (in $d$-metric) and 
\[
\Omega = \{(x_1,x',t) \in B_d(0,1) : -1 < x_1 < \psi(x',t) \},
\] 
then $\Omega\subset\mathbb{R}^{n+2}$ is a $(t_0,A)$-FIT H\"older domain for $t_0 \in (0,1)$.

\subsection*{Example 2: Failure of $t_0$-connectivity}
A very nice domain in terms of $d$-metric can fail the connectivity.
For instance, if $\Omega \subset \R^{2}$ is the interior of the sum-set
\[
[-1/2,1/2]^{2} + \{(-1,0),(-1,1),(0,1),(1,1), (1,0) \},
\]
we have an example of a domain failing the $t_0$-FIT connectivity for all $t_0 \le 1/2$.
 
\subsection*{Example 3:  Failure of $(t_0,A)$-H\"older with $(t_0,A)$-connectivity holding}
A typical way to fail the accessibility with a $(t_0,A)$-H\"older condition without violating the $t_0$-connectivity condition is a purely lateral part of boundary.
For instance, if $\Omega \subset \R^{2}$ is the interior of the sum-set
\[
[-1/2,1/2]^{2} + \{(-1,0),(-1,1),(0,1) \},
\]
we have an example of a domain satisfying the $t_0$-FIT connectivity for all $-1/2 < t_0 < 3/2$
but failing the bound \eqref{eq:defn:FITHolder} (for all possible constants) when $t_0 \le 1/2$.
This is similar to the fact that parabolic chord arc domains in the plane are actually graph domains (see \cite{Eng17}).
Further,
we point out that in this domain the parabolic forward-in-time BMO function $k$ from Lemma \ref{lem:universal-pbmo}
has a power type (as opposed to logarithmic) blow-up towards the boundary point $(-1,0) \times \{1/2\}$,
as is to be expected from Theorem \ref{theorem:intro1}.

\subsection{Example 4: Failure of  spatial boundedness with $(t_0,A)$-H\"older holding}\label{sect:examples:ex4}
Choosing a constant $M > 1$ depending on the parameter $p \in (1,\infty)$ fixing the metric
large enough,
with
\[
\Omega = \{(x,t) \in \R^{2}: t < e^{-M|x|}  \} ,
\]
we have an example of a domain that is $t_0$-FIT connected for all $t_0 < 1$
and which is $(t_0,A)$-FIT H\"older for $t_0 \in [0,1)$. 
This domain is not $(t_0,A)$-FIT H\"older for $t_0  < 0$ (that would contradict the finite volume of $(t_0,A)$-FIT H\"older sets),
and
setting $t_0 = 0$ we see that $(t_0,A)$-FIT H\"older allows for unboundedness.
Indeed, for $(x,t) \in \Omega_{0}^{+}$ there holds that
\[
\log_{2}\biggl(1+ \frac{1}{\dist_d((x,t), \Omega^{c})} \biggr) \gtrsim 1 + |x|^{M}.
\]
The relevant FIT chains in this $\Omega$ are essentially as in the free space,
so \eqref{eq:defn:FITHolder} is easily satisfied.

\bibliography{references}

\end{document}